\documentclass[11pt]{amsart}
\pdfoutput=1

\usepackage{amsfonts, amstext, amsmath, amsthm, amscd, amssymb}
\usepackage{epsfig, graphics, psfrag, overpic, color}
\usepackage{enumerate}

\newtheorem{thm}{Theorem}[section]
\newtheorem{lem}[thm]{Lemma}

\newtheorem{cor}[thm]{Corollary}

\newtheorem*{rmk}{Remark}
\newtheorem{defn}[thm]{Definition}

\newtheorem*{mainthm}{Theorem \ref{thm:pattern1}}
\newtheorem*{maincor}{Corollary \ref{cor:Wd}}
\newtheorem*{thmtw}{Theorem \ref{thm:twist}}
\newtheorem*{3braids}{Corollary  \ref{cor:3brds}}
\newtheorem*{corfibered}{Corollary \ref{cor:twist1}}

\def \Z{\mathbb{Z}}
\def \T{\mathcal{T}}
\def \K{\mathbb{K}}

\def \C{\mathcal{C}}

\begin{document}

\title[Cosmetic crossings of twisted knots]{Cosmetic crossings of twisted knots}
\author[C. J. Balm and E. Kalfagianni]{Cheryl Jaeger Balm and Efstratia Kalfagianni}

\address[]{Department of Mathematics, Kansas State
University, Manhattan, KS 66502}

\email[]{cjbalm@math.ksu.edu}

\address[]{Department of Mathematics, Michigan State
University, E Lansing, MI, 48824} \ \ \
\email[]{kalfagia@math.msu.edu}

\thanks{ {Research supported by NSF grant DMS-1105843.
}}

\begin{abstract}  We prove that the property of admitting no cosmetic crossing changes 
is preserved under the operation of forming certain satellites of winding number zero. 
We also define strongly cosmetic crossing changes and we discuss their
behavior under the operation of inserting full twists in the strings of closed braids.

%
%
%
\end{abstract}

\maketitle

\bigskip


\section{Introduction}

 A \emph{crossing disk} for an oriented knot $K\subset S^3$ is an embedded disk $D\subset S^3$ such that $K$ intersects ${\rm int}(D)$ twice with zero algebraic intersection number.  A crossing change on $K$ can be achieved by performing $(\pm 1)$-Dehn surgery of $S^3$ along the \emph{crossing circle} $L = \partial D$.  More broadly,  a  \emph{generalized crossing change} of order $q \in \Z - \{ 0 \}$
 is achieved by $(-1/q)$-Dehn surgery along the crossing circle $L$ and results in introducing $q$ full twists to $K$ at the \emph{crossing disk} $D$ bounded by $L$.  
A (generalized) crossing change of $K$ and its corresponding crossing circle $L$ are called \emph{nugatory} if $L$ bounds an embedded disk in $S^3 - \eta(K)$, where $\eta(K)$ denotes a regular neighborhood of $K$ in $S^3$.  Obviously, a generalized crossing change of any order at a nugatory crossing of $K$ yields a knot isotopic to $K$.  
 \begin{defn}{\rm
A (generalized) crossing change on $K$ and its corresponding crossing circle are called \emph{cosmetic} if the crossing change yields a knot isotopic to $K$ and is performed at a crossing of $K$ which is \emph{not} nugatory.}
\end{defn}

%

It is an open question whether there exist knots that admit cosmetic crossing changes.
(See Problem 1.58 on Kirby's list \cite{kirby}.)
This question, often referred to as the \emph{nugatory crossing conjecture},   has motivated quite a bit of research over the years, and  it has been answered in the negative for many classes of knots.  Scharlemann and Thompson showed that the unknot admits no  cosmetic generalized crossing changes in \cite{schar-thom} using work of Gabai \cite{gabai}.  It has been shown by Torisu that the answer is also no for 2-bridge knots \cite{torisu},
and  by the second author that the answer is no for fibered knots \cite{kalf}. Obstructions to cosmetic crossing changes in genus-one knots were found by the authors with Friedl and Powell in \cite{balm}, where it is shown that genus-one, algebraically non-slice knots admit no cosmetic generalized crossing changes. The objective of the current paper is to study the behavior of potential cosmetic crossing changes under the operations of twisting and forming satellites.

To state our results, let $\K$ denote the class of knots which do not  admit cosmetic generalized crossing changes.  By the previous paragraph, $\K$ contains all fibered knots, 2-bridge knots and genus-one, algebraically non-slice knots.  Torisu shows in \cite{torisu} that the connect sum of two or more knots in $\K$ is also in $\K$.

\begin{defn}\label{def:wrap}{\rm Given a knot $K$ embedded in a solid torus $V$, the \emph{winding number}, $w(K,V)$, is the  algebraic intersection number of $K$ with a meridional disk of $V$.  Similarly, the \emph{wrapping number}, $\textrm{wrap}(K,V)$, is the minimal geometric intersection of $K$ with a meridional disk. } 
\end{defn}

In \cite{balm2}, the first author shows that any prime satellite knot with pattern a non-satellite knot in  $\K$ does not admit cosmetic
generalized  crossing changes of order greater than $5$. Here we restrict ourselves to satellites with winding number zero, and we obtain the following stronger result.

\begin{thm}\label{thm:pattern1} Let $C$ be a prime knot that is not a torus knot or a cable knot and let $V'$ be a standardly embedded solid torus in $S^3$.  Let $K'\in \K$  and suppose that $K'$ is embedded in $V'$ so that it is geometrically essential and such that  $w(K',V')=0$ and $V'\setminus  \eta(K')$ is atoroidal.  Then any knot that is a satellite of $C$ with pattern $(V',  K')$  admits no  cosmetic generalized crossing changes of any order.
\end{thm}

Theorem \ref{thm:pattern1} has the following corollary, which generalizes  Corollary 7.1(b)  of  \cite{balm}, where only ordinary crossing changes in
twisted Whitehead doubles are treated.

\begin{cor}\label{cor:Wd} Let $K$ be a prime knot that is not a torus knot or a cable knot. 
Then no Whitehead double of $K$ admits a cosmetic generalized crossing change of any order.
\end{cor}

\begin{defn}\label{def:twistK}
{\rm  We will say a torus $T$ is \emph{standardly embedded} in $S^3$ if $T$ bounds a solid torus on both sides. Let $V$ be a solid torus standardly embedded in $S^3$.  Let $K$ be a knot that is a closed $m$-braid in $V$, in the sense that  $m=w(K,V) = { \rm{wrap}}(K, V) $ and $K$ intersects each meridian disk exactly $m$-times.

Given  $n\in \Z$, let $K_{n, V}$ denote the image of $K$ under the $n^{\rm th}$ power of a meridional Dehn twist of 
$V$. In other words, $K_{n, V}$ is obtained from $K$ via 
$n$ full twists along a meridian disk $D \subset V$. Note that if $n=0$, $K_{n, V}$ is simply the given embedding of $K$ in $V$. When there is no danger of confusion, we will simply write $K_n$ instead of $K_{n, V}$. 

Let $K_{n, V}$ be as above and let $K_{n, V}(q)$ denote a knot obtained from $K_{n, V}$
by an order-$q$ crossing change in $V$. We say that the crossing  change is \emph{strongly cosmetic}
if $K_{n, V}$ and $K_{n, V}(q)$ are isotopic knots in $V$.}
\end{defn}

\begin{thm}\label{thm:twist}
Let $K' \in \K$ be embedded as   an $m$-braid in  a  standard solid torus $V'$. Then, for every $n\in \Z$,  the twist knot  $K'_{n, V'}$  does not admit a strongly cosmetic generalized crossing change of any order.
\end{thm}

Theorem \ref{thm:twist} together with the results of \cite{kalf} imply that twist knots of fibered closed braids do not admit
cosmetic crossing changes of any order. To state this result more precisely, for $m>0$, let $B_m$ denote the $m$-string braid group,
and let $\Delta_m^2$ denote the central element in $B_m$ (the full twist braid). 

\begin{defn}{\rm
A \emph{fibered  $m$-braid} is a closed braid on $m$ strands which is also a fibered knot. Given a fibered $m$-braid $K$ and  $n\in \Z$, the knot  $K_{n}$  obtained by inserting a  copy of $\Delta_m^{2n}$ ($n$ full twists) into the strings of $K$ is called a \emph{twisted fibered  braid}.}
\end{defn}

As a corollary of Theorem \ref{thm:twist} we obtain the following.

 \begin{cor}\label{cor:twist1} Let $K$ be a fibered $m$-braid with $m\geq 2$.  Then for every $n\in \Z$, the twisted fibered braid $K_{n}$  admits no strongly  cosmetic crossing changes of any order.
\end{cor}

Knots that are closures of braids on three strands fit into the setting of Corollary \ref{cor:twist1}.  This is more fully discussed in Section \ref{s:pfs} and leads to the following result.

\begin{cor} \label{cor:3brds} Let $K$ be a knot  that can be represented as the closure of a 3-braid. Then $K$ admits no strongly cosmetic crossing changes of any order.
\end{cor}

 This paper is organized as follows.  In Section \ref{s:background} we give the necessary definitions and lemmas to prove Theorems \ref{thm:pattern1} and \ref{thm:twist}.  Then in Section \ref{s:pfs} we prove these results and their corollaries.
\newpage

\section{Crossing circles and companion tori}\label{s:background}

We begin by recalling some definitions.

\begin{defn}{\rm Let $V'$ be a standardly embedded solid torus in $S^3$, and let $K'$ be a knot embedded in $V'$ so that $K'$ is geometrically essential in $V'$ and not
the core of $V'$.  Let  $f:(V', K') \to  S^3$ be an embedding such that $V=f(V')$ is a knotted solid torus in $S^3$.  A \emph{satellite knot} with \emph{pattern} $K'$ is the image $K = f(K')$.  If $C$ is the core of the solid torus $V$, then $C$ is a \emph{companion knot}
 of $K$, and we may call $K$ a \emph{satellite of $C$}.  The torus $T=\partial V$ is a \emph{companion torus} of $K$.  We may similarly define a \emph{satellite link} if $K'$ is a non-split link.}
\end{defn}

Given a 3-manifold $N$ and a submanifold $F \subset N$ of co-dimension 1 or 2, $\eta (F)$ will denote a regular neighborhood of $F$ in $N$.  For a knot or link $K \subset S^3$, we define $M_K = \overline{S^3 - \eta(K)}$.  

Given a knot $K$ with a crossing circle $L$, let $K(L,q)$ denote the knot obtained via an order-$q$ generalized crossing change at $L$.  We will simply write $K(q)$ for $K(L,q)$ when there is no danger of confusion about the crossing circle in question.  We will also use the notation $K(0)$ when we wish to be clear that we are referring to the embedding of $K$ in $S^3$ before any crossing change occurs.

In general, given a knot $K$ and a crossing circle $L$ for $K$, let $M(q)$ denote the 3-manifold obtained from $M_{K \cup L}$ via a Dehn filling of slope $(-1/q)$ along $\partial \eta (L)$.  So for $q \in \Z - \{ 0 \},\ M(q) = M_{K(q)}$ and $M(0) = M_K$.  

The first lemma which we will need in the proof of the results stated in the introduction is the following.

\begin{lem}\label{lem:LinV}
Let $K$ be a satellite knot, $T$ be a companion torus for $K$,  and $V$ be the solid torus bounded by $T$ in $S^3$.
Suppose that  $w(K, V) = 0$ and that there are no essential annuli in $\overline{S^3 - V}$.   If $L$  is a cosmetic crossing circle
 for $K$ we can isotope it so that $L $ and a crossing disk bounded by $L$ both lie in $V$.
\end{lem}

\begin{proof}
Let $K$ be as in the statement of the lemma, and suppose that $L$ is an order-$q$ cosmetic crossing circle.  Let $S$ be a minimal genus  Seifert surface for $K$ in $\overline{S^3 - \eta(L)}$, and let $D$ be the crossing disk for $K$ which is bounded by $L$.  We may choose  $S$ so that $S \cap D$ is a single embedded arc $\alpha$.  Then performing $(-1/q)$-surgery at $L$ twists both $K$ and $S$, and produces a surface $S(q) \subset M(q)$ which is a Seifert surface for $K(q)$.  

Note that if $M_{K \cup L}$ were reducible, then $M_{K \cup L}$ would contain a separating 2-sphere which does not bound a 3-ball $B \subset M_{K \cup L}$.  Then $L$ would lie in a 3-ball disjoint from $K$; hence $L$ would bound a disk in this 3-ball, which is in the complement of $K$, so $L$ would be nugatory.  Since $L$ is cosmetic by assumption, we may conclude that $M_{K \cup L}$ is, in fact, irreducible.  Thus we may apply Gabai's Corollary 2.4 of \cite{gabai} to see that $S$ and $S(q)$ are minimal genus Seifert surfaces in $S^3$ for $K$ and $K(q)$, respectively. 

If $S \subset V$, then we may shrink $D$ to a neighborhood of $\alpha$ which is orthogonal to $S$, and hence $D \subset V$.  Assume that $S \not\subset V$, and let $\C = S \cap T$.  We may isotope $S$ so that $\C$ is a collection of simple closed curves which are essential in both $S$ and $T$.  Since $w(K,V)=0$, $\C$ must be homologically trivial in $T$, where each component of $\C$ is given the boundary  orientation  $S\cap V$.  Hence $\C$ bounds a collection of annuli in $T$ which we will denote by $A_0$. 

Let $S_0 = S - (S \cap V)$.  Suppose that $\chi (S_0) < 0$, where $\chi( \cdot )$ denotes the Euler characteristic.  We may create $S^*$ from $S$ by replacing $S_0$ by $A_0$, isotoped slightly, if necessary, so that each component of $A_0$ is disjoint.  Then $S^*$ is a Seifert surface for $K$, and $\chi (S^*) > \chi (S)$ since $\chi (A_0) = 0$.  This contradicts the fact that $S$ is a minimal genus Seifert surface for $K$, so it must be that $\chi (S_0) \geq 0$.  Since $S_0$ contains no closed component, and no component of $\C$ bounds a disk in $S$, we conclude that $S_0$ consists of annuli.

By assumption, there are no essential annuli in $\overline{S^3 - V}$, so each component of $S_0$ must be boundary parallel in $\overline{S^3 - V}$.  Thus we can isotope $S_0$ so that $S \subset V$, and therefore $D$ can be isotoped into $V$ as well. \end{proof}

This leads us to the following lemma, used in the proof of Theorem \ref{thm:twist}.

\begin{lem}\label{lem:LinVtwist} 
 Let $V'$ be a solid torus standardly embedded in $S^3$, let $K'$ be a knot that is 
an $m$-braid in $V'$, and let $K=K'_n$ be a twist knot of $K'$ for some $n \in \Z$.  If $L$ is a cosmetic crossing circle of $K$, then we can isotope it so that $L $ and a crossing disk bounded by $L$ lie in $V$.
\end{lem}

\begin{proof} Suppose that the cosmetic crossing change of $K$ is of order $q$.  Let
 $S$ be a minimal genus  Seifert surface for $K$ in $\overline{S^3 - \eta(L)}$, and let $D$ be the crossing disk for $K$ which is bounded by $L$. 
 As in the proof of Lemma \ref{lem:LinV},
 we  may isotope $S$ so that $S \cap D$ is a single embedded arc $\alpha$.  Then performing $(-1/q)$-surgery at $L$ twists both $K$ and $S$, producing a surface $S(q) \subset M(q)$ which is a Seifert surface for $K(q)$.  As in the proof of Lemma \ref{lem:LinV},
 we conclude that $S$ and $S(q)$ are minimal genus  surfaces in $S^3$ for $K(0)$  and $K(q)$, respectively. 
 
Let $W$ be the solid torus $\overline{S^3-V'}$.  Since $S$ in minimal genus, each component of $S \cap W$ is incompressible in $W$ and therefore is either a disk or an annulus which is parallel to an annulus in $T$.  We can isotope $S$ to remove the annular components of $S \cap W$, so we may assume that each component $C\in T\cap S$  bounds a disk $D_C \subset S$ in the complement of $V$. For each such disk, the intersection  $\alpha\cap D_C$ is a collection of properly embedded arcs in $D_C$. Each of these arcs can be  isotoped onto $\partial D_C\subset T$ by an isotopy on $D_C$ relative the boundary points of $\alpha$.  This process will isotope $\alpha$ into $V$ by an isotopy of $S$ which fixes $\partial S = K$. Since $L$ lies in small neighborhood of $\alpha$, by further isotoping each arc of $\alpha \cap T$ into $\rm{int}(V)$, we may bring $L$ and $D$ into $V$, as desired. 
  \end{proof}

We will also need the following lemma, which is proved by arguments similar to those in the proofs of Lemmas \ref{lem:LinV} and \ref{lem:LinVtwist}.  
\begin{lem}[Lemma 4.6 of \cite{kalf-lin}]\label{lem:KinV}
Let $V \subset S^3$ be a knotted solid torus such that $K \subset \rm{int}(V)$ is a knot which is essential in $V$ and $K$ has a crossing disk $D$ with $D \subset \rm{int}(V)$.  If $K$ is isotopic to $K(q)$ in $S^3$, then $K(q)$ is also essential in $V$.  Further, if $K$ is not the core of $V$, then $K(q)$ is also not the core of $V$.
\end{lem}

We close the section with the following lemma, which discusses the interplay of nugatory crossing changes with satellite operations.

\begin{lem}\label{lem:nugasat} Let $K'$ be a  knot that is essentially embedded in a standard solid torus $V'$ and let $L'$ be a crossing circle for $K'$ that lies in $V'$. Suppose that there is an orientation preserving homeomorphism
$h: V' \to V'$ that takes the preferred longitude of $V'$ to itself and such that $h(K'(q))=K'(0)$.  Then if  $L'$ is nugatory for $K'$ in $S^3$,  it is also
nugatory for $K'$ in $V'$. 
\end{lem}

\begin{proof} Note that the existence of  $h$ as in the statement above implies that $K'(q)$  and $K'$ are ambiently isotopic in $V'$. 
Suppose $L'$ is nugatory.  Then $L'$ bounds a crossing disk $D'$ in $V'$ and another disk $D''$ in the complement of $K'$.  We may assume $D' \cap D'' = L'$.  

Let $A_{V'}: = D' \cup (D'' \cap V')$. Now $A_{V'}$
contains a component  $(A, \partial A)\subset  (V', \partial V')$  that is a
a properly embedded  planar surface in $V'$, so each component of  $\partial A $ is a preferred longitude of $V'$, and we have $D'\subset A$.

 Extend the homeomorphism $h$ on $V' \subset S^3$ to a homeomorphism $H$ on all of $S^3$.  After an isotopy that fixes
 $h(K'(q))=K'(0)$ 
  we may assume $H(D')$ is isotopic to $D'$ and $H(D'')$ is isotopic to $D''$.  In fact, this isotopy may be chosen so that $H(V')=V'$ still holds after the isotopy.  Thus, we may assume $h(A)=A$ and $A$ cuts $V'$ into two components, $V'_1$ and $V'_2$.  An example where $A$ is an annulus is shown in Figure \ref{fig:torus}.

\begin{figure}
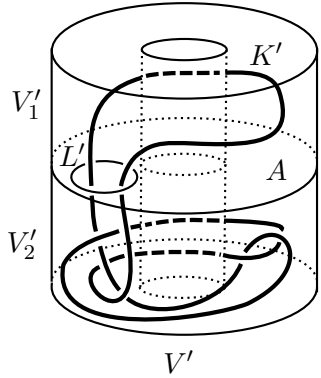
\caption{The solid torus $V'$, cut into two components by $A$, as in the proof of Lemma \ref{lem:nugasat}.}\label{fig:torus}
\end{figure}

First suppose that $ h( V'_i )= V'_i$ for $i=1,2$.  

The sphere $D' \cup D'' $ separates $S^3$ into two 3-balls, $B_1$ and $B_2$. Assume that  $V'_i \subset B_i$ for $i=1,2$.
Note that in $S^3$  there is an isotopy from $K'(q)$ to $K'(0)$ that untwists  the part of $K'(q)$ outside  of 
$B_1$ and leaves fixed  the part of $K'(q)$ that  lies outside a collar neighborhood of the boundary of $B_2$.  This isotopy gives an orientation-preserving homeomorphism
 $f: S^3 \longrightarrow S^3$ which brings  $K'(q)$ to $K'(0)$ .
Using $h$ above we identify $V'_i$ with $h(V'_i)$ and $K'(0)$ with $K'(q)$ and simply denote it by $K'$.

Let $X$ be the 3-manifold obtained from $V'_2$ by drilling out a neighborhood of $K' \cap V'_2$. 
Now  $f$ restricted to $X$ 
 is a homeomorphism whose   restriction to $\partial V'_2$ is a Dehn twist along $L'$. 
  By a result of McCullough (Theorem 1 of \cite{mccullough}), $L'$ bounds a disk in $V'_2 \subset (V' - \eta(K'))$.  Thus $L'$ is nugatory for $K'$ in $V'$.

If $ h$ maps $ V'_2$ to $V'_1$ and vice versa, then we may use a similar argument.  In this case, $f$ untwists the part of $K'(q)$ in $B_1$, leaving fixed the part of $K'(q)$ outside a collar neighborhood of $\partial B_1$, and we identify $V_2'$ with $h(V'_1)$.


\end{proof}

\smallskip

\section{Proofs of main results}\label{s:pfs}
We are now ready to prove the results stated in the introduction.

\subsection{Satellites} First, we discuss satellite knots of winding number zero, and we prove Theorem \ref{thm:pattern1}.  Recall that $\mathbb{K}$ is the class of knots which do not admit cosmetic generalized crossing changes.

\begin{mainthm}
Let $C$ be a prime knot that is not a torus knot or a cable knot and let $V'$ be a standardly embedded solid torus in $S^3$.  Let $K'\in \K$  and suppose that $K'$ is embedded in $V'$ so that it is geometrically essential and such that  $w(K',V')=0$ and $V'\setminus  \eta(K')$ is atoroidal.  Then any knot that is a satellite of $C$ with pattern $(V',  K')$  admits no  cosmetic generalized crossing changes of any order.
\end{mainthm}

\begin{proof}
Let $(V',K')$ be as in the statement of the theorem and consider the satellite map $f:(V',K') \to (V,K)$ with $\textrm{core}(V)=C$.  Suppose that $K$ admits an order-$q$ cosmetic crossing change, and let $D$ be the corresponding crossing disk with $L = \partial D$.   Let $T = \partial V$. 

 Since $C$ is a prime knot that is not a torus knot or a cable knot,
there are no essential annuli in $\overline{S^3 - V}$ (see, for example, Lemma 2 of \cite{lyon}). Hence, by Lemma \ref{lem:LinV}, we may assume $D \subset V$, so $T$ is also a companion torus for the satellite link $K \cup L$.

Now $K' \cup L'$ is a pattern link for $K \cup L$ with the satellite map $f: (V', K', L') \to (V,K,L)$ as above.  We will show that $L'$ is an order-$q$ cosmetic crossing circle for $K'$, which is a contradiction since $K' \in \K$.  

Since $L$ is cosmetic, $M = M_{K \cup L}$ is irreducible.  Consider a finite collection of tori  $\T$  for $M$ with the properties that $T\in \T$, the tori in $\T$ are essential, no two tori in $\T$ are parallel in $M$, and each component of $M$ cut along $\T$ is atoroidal.
By Haken's Finiteness Theorem (Lemma 13.2 of \cite{hempel})  such a collection exists.

We will call a torus $F \in \T$  \emph{innermost with respect to $K$} if $M$ cut along $\T$ has a component $N$ such that $\partial N$ contains $\partial \eta (K)$ and a copy of $F$.  In other words, $F \in \T$ is innermost with respect to $K$ if there are no other tori in $\T$ separating $F$ from $\eta(K)$.  
\vskip 0.06in

{\it{ Case 1.} }$T$ is innermost in $\T$ with respect to $K$.

Let $W= \overline{V - \eta(K \cup L)}$.  We first wish to show that $W$ is atoroidal.  By way of contradiction, suppose that there is an essential torus $F \subset W$.  Then $F$ bounds a solid torus in $S^3$, which we will denote by $\widehat{F}$.  

\vskip 0.04in

{\it Subcase 1.} $\widehat{F}$ is contained in $V$.

Since $T$ is innermost with respect to $K$, either $F$ is parallel to $T$ in $M$, or $K \subset V - \widehat{F}$.  
By assumption, $F$ is essential in $W$ and hence not parallel to $T \subset \partial W$.  So $K \subset V - \widehat{F}$ and, since $F$ is incompressible, $L \subset \widehat{F}$.  

If $\widehat{F}$ is knotted, then either $L$ is the core of $\widehat{F}$ or $L$ is a satellite knot with companion torus $\partial \widehat{F}$.  This contradicts the fact that $L$ is unknotted.  Hence $F$ is an unknotted torus.  By definition, $L$ bounds a crossing disk $D$.  Since $D$ meets $K$ twice, $D \cap \textrm{ext}(\widehat{F}) \neq \emptyset$.  We may assume that $D$ has been isotoped (rel boundary) to minimize the number of components in $D \cap F$.  Since an innermost component of $D - (F \cap D)$ is a disk and $L$ is essential in the unknotted solid torus $\widehat{F}$, $D \cap F$ consists of standard longitudes on the unknotted torus $F$.  Hence $D \cap \textrm{ext}(\widehat{F})$ consists of either one disk which meets $K$ twice, or two disks which each meet $K$ once.  In the first case, $L$ is isotopic to the core of $\widehat{F}$, which contradicts $F$ being essential in $W$.  In the latter case, the linking number $\textrm{lk}(K,\widehat{F})= \pm 1$.  So $K$ can be considered as the trivial connect sum $K \# U$, where $U$ is the unknot, and the crossing change at $L$ takes place in the unknotted  summand $U$.  
(See Figure \ref{fig:unknottedF}.)  
The unknot does not admit cosmetic crossing changes of any order by \cite{schar-thom}, so $K(q) \cong K \# K'$ where $K' \not\cong U$.  This contradicts the fact that $K(q) \cong K$. 

\begin{figure}
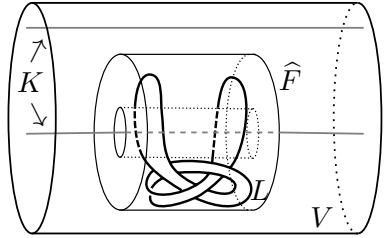
\caption{A portion of the solid torus $V$ containing the unknotted solid torus $\widehat{F}$ from Case 1, Subcase 1 of the proof of Theorem \ref{thm:pattern1}. }\label{fig:unknottedF}
\end{figure}

{\it Subcase 2.} $\widehat{F}$ is  not contained contained in $V$.   

Then $\widehat{F}$ is a knotted solid torus and the complement $S^3 \setminus \widehat{F}$ is contained in $V$.  This complement is a 3-ball with a knotted 1-handle drilled out, so $K$ must be contained in $\widehat{F} \cap V$.  (See Figure \ref{fig:figureA} for an example.)  Since $V' \setminus \eta(K')$ is atoroidal, $V \setminus \eta(K)$ is too; hence $L$ is also contained in $\widehat{F} \cap V$.

\begin{figure}
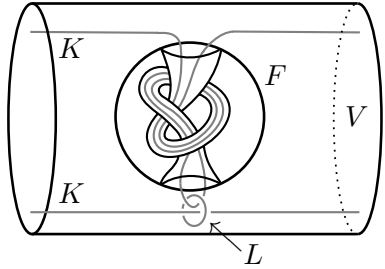
\caption{The knotted torus $F$ from Case 1, Subcase 2 of the proof of Theorem \ref{thm:pattern1}.  A generalized crossing change of order $q$
will make $F$ essential in the complement of $K(q)$. }\label{fig:figureA}
\end{figure}

The torus  $F$ becomes inessential in the complement of $K(0)$.  
On the other hand $F$ becomes essential in the complement of $K(q)$. This is, however, impossible since $L$ was assumed to be a cosmetic crossing circle, and thus the complement of $K(0)$ is homeomorphic to that of $K(q)$. Therefore this case will not happen.

Together Subcases 1 and 2 imply that $W$ is atoroidal in Case 1. Hence, $W' = \overline{V' - \eta(K' \cup L')}$ must be atoroidal as well.

If $K(q)$ is not geometrically essential in $V$, then, by Lemma \ref{lem:KinV}, $K(0)$ is also not essential in $V$.  But this contradicts $V$ being a companion for $K$, so $K(q)$ must be essential in $V$.  Hence $T$ is a companion torus for both $K(q)$ and $K(0)$.  Since $L$ is cosmetic, there is an ambient isotopy $\psi: S^3 \to S^3$ taking $K(q)$ to $K(0)$ such that $V$ and $\psi(V)$ are both knotted solid tori containing $K(0) = \psi(K(q)) \subset S^3$.  Then,  Lemma 2.3 of \cite{motegi} applies to conclude that there is an ambient  isotopy $\phi: S^3 \to S^3$,
 fixing $K(0)$, such that, if we let $\Phi = (\phi \circ \psi)$,
  one of the following holds:
 \begin{enumerate}
\item  $\Phi(T) \cap T = \emptyset$.
\item There exist disjoint meridian disks $D$ and $D'$ for both $V$ and $\Phi(V)$ such that some component of $V$ cut along $(D \sqcup D')$ is a knotted 3-ball in some component of $\Phi(V)$ cut along $(D \sqcup D')$.  (See, for example, Figure \ref{fig:knottedball}.)
\end{enumerate}

\begin{figure}
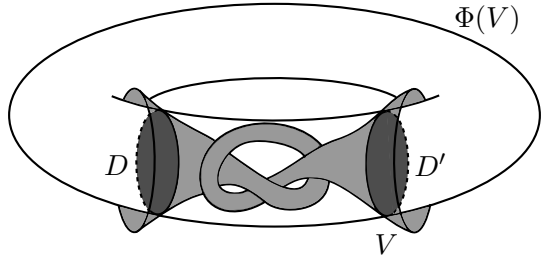
\caption{A knotted 3-ball created by the merdian disks $D$ and $D'$.}\label{fig:knottedball}
\end{figure}
 
Option (2) would imply the existence of an essential torus in $W$.
Since $W$ is atoroidal, we conclude that $(\phi \circ \psi)(T) \cap T = \emptyset$. Since $T$ is innermost with respect to $K$, $\Phi(T) \cap T = \emptyset$ implies $T$ and $\Phi (T)$ are parallel in $M_K$.  So, after an isotopy which fixes $K(0) \subset S^3$, we may assume that $\Phi(V) = V$.

Let $h = (f^{-1} \circ \Phi \circ f):V' \to V'$.  Then $h$ maps $K'(q)$ to $K'(0)$, and hence $K'(q)$ and $K'(0)$ are isotopic in $S^3$.  So either $L'$ gives an order-$q$ cosmetic generalized crossing change for the pattern knot $K'$, or $L'$ is a nugatory crossing circle for $K'$. Since $K' \in \K$,  $L'$ has to be nugatory. By Lemma \ref{lem:nugasat}, $L$ is nugatory for $K$, which contradicts our assumption that $L$ is cosmetic.

\vskip 0.05in

{{\it { Case 2.}} }$T$ is not innermost in $M$ with respect to $K$.

 Let $F$ be a torus in $\mathcal{T}$ which it is ``closer" to $K$ than $T$.  Since we assumed $V'\setminus  \eta(K')$ is atoroidal, $F$ must become inessential in the complements of $K(0)$ and $K(q)$ in $V$. Take $F$ to be a torus that is innermost to $K$ in $\mathcal{T}$. Then arguments similar to those in Subcases 1 and 2 above imply this $F$ cannot exist if $L$ is cosmetic.
\end{proof}

To conclude, we restate and prove Corollary \ref{cor:Wd} from Section 1.

\begin{maincor}
Let $K$ be a prime knot that is not a torus knot or a cable knot. 
Then no Whitehead double of $K$ admits a cosmetic generalized crossing change of any order.
\end{maincor}

\begin{proof}
All Whitehead doubles admit a pattern $(V',U)$ where $U$ is the unknot, $w(U,V')=0$, and $U$ intersects every meridian disc of $V'$ at least twice.  Furthermore  $V' \setminus  \eta(U)$ is atoroidal.  Hence, the result follows immediately from Theorem \ref{thm:pattern1}.

\end{proof}

\subsection{Twisted knots} We now give the proofs of Theorem \ref{thm:twist} and Corollary \ref{cor:3brds}, which we also restate for convenience.

\begin{thmtw} 
Let $K' \in \K$ be embedded as   an $m$-braid in  a  standard solid torus $V'$.
Then, for every $n\in \Z$,  the twist knot  $K'_{n, V'}$  does not admit a strongly  cosmetic generalized crossing change of any order.

\end{thmtw}

\begin{proof}
Suppose that for some $K' \in \K$ there is an embedding of $K'$ into a standard solid torus $V'$ as in the statement of the theorem, and that for some $n \in \Z$, the twist knot $K=K'_{n,V'}$ admits  an order-$q$ strongly cosmetic crossing change corresponding to a crossing circle $L$.
That is, $K(0)$ and $K(q)$ are isotopic in $S^3$.  Let $f: V'\rightarrow V=f(V')$ denote the twisting homeomorphism bringing $K'$ to $K$.

By Lemma \ref{lem:LinVtwist}, 
we may isotope $L$ into $V$. 
Now $L$ pulls back, via  $f$, to a crossing circle $L'$ of $K'$ in $V'$, and the generalized crossing change 
on $K$ pulls back to a generalized crossing change on $K'$. Let $K'(q)$ denote the result of this crossing change on $K'$.

Since the winding number doesn't change by twisting along a disc, we have  $0 \neq w(K,V)=w(K(q), V)=m$.  Thus both $K$ and $K(q)$ are essential in $V$.
By definition, there  is an orientation-preserving diffeomorphism $\Phi: S^3 \longrightarrow S^3$ that brings $K(q)$ to  $K=K(0)$ and $\Phi(V) =V$.
Let $h= (f^{-1} \circ \Phi \circ f):V' \to V'$.  Then $h$ maps $K'(q)$ to $K'$ and hence $K'(q)$ and $K'$ are isotopic in $S^3$. 
So either $L'$ gives an order-$q$ cosmetic generalized crossing change for $K'$, or $L'$ is a nugatory crossing circle for $K'$.
Since $K' \in \K$,  $L'$ has to be nugatory.  Since $h$ maps the preferred longitude of $V'$ to itself, $K'(q)$ and $K'$ are isotopic in $V'$.
By Lemma \ref{lem:nugasat},  $L'$ is nugatory for $K'$ in $V'$. But since $f$ is homeomorphism of the solid torus,
$L=f(L')$ is also nugatory for  $K=K'_{n,V'}$, which contradicts our assumption that $L$ is cosmetic.
\end{proof}

We are now ready to prove Corollary \ref{cor:twist1}.

 \begin{corfibered}Let $K$ be a fibered $m$-braid with $m\geq 2$. 
 Then for every $n\in \Z$, the twisted fibered braid $K_{n}$
 admits no  strongly cosmetic crossing
changes of any order.
\end{corfibered}

\begin{proof} Since $K$ is a fibered knot, by \cite{kalf},  we have $K' \in \K$.  Embed $K$ as a closed $m$-braid in a standardly embedded solid torus $V'$.
Then, for every $n\in \Z$, the twisted fibered braid $K_{n}$ is the result of $K$ under an order $n$ Dehn  twist  of $V$ along $\partial D$.
Thus $K_{n}$ is a twist knot of $K$ in the sense of  Definition \ref{def:twistK}, and the conclusion follows from
Theorem \ref{thm:twist}.
\end{proof}

\begin{rmk} {\rm  One may consider closures of \emph{positive} or \emph{homogeneous}  braids, which are known to be fibered by \cite{stallings}, and then add full twists to obtain broad classes  twisted fibered braids.}
 \end{rmk}

Next we discuss knots obtained as 3-braid closures. Let $\sigma_1$ and $\sigma_2$ denote the standard 
generators for $B_3$, the braid group on three strands.  (See Figure \ref{fig:3braid-generators}.) A positive word in $\sigma_1$ and
$\sigma_2^{-1}$ represents an alternating braid diagram. Let $\Delta^2=
(\sigma_1 \sigma_2)^3$ denote a full twist of all three strands, so $\Delta^2$
generates the center of $B_3$. For a braid $w \in B_3$, let
$\hat{w}$ denote the knot or link obtained as the standard closure of $w$. Note that
$\hat{w}$ only depends on the conjugacy class of $w$. 

\begin{figure}
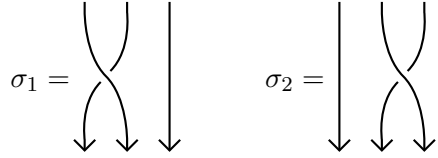
\caption{The generators $\sigma_1$ and $\sigma_2$ of the 3--string braid group.}\label{fig:3braid-generators}
\end{figure}

\begin{3braids} Let $K$ be a knot that can be represented as the closure of a 3-braid. Then $K$ admits no
strongly  cosmetic crossing changes of any order.
\end{3braids}

\begin{proof} Suppose that knot $K=\hat{w}$ for some $w\in B_3$.
By  Schreier \cite{schreier}, $w$ is conjugate to a
braid in exactly one of the following forms:
\begin{enumerate}
\item $\Delta^{2k} \sigma_1^{p_1} \sigma_2^{-q_1} \cdots \sigma_1^{p_s}
\sigma_2^{-q_s}$, where $k \in \Z$ and $p_i$, $q_i$, and $s$ are all
positive integers
\item $\Delta^{2k} \sigma_1^p\ $ for some $k, p \in \Z$
\item $\Delta^{2k} \sigma_1 \sigma_2\ $ for some $k \in \Z$
\item $\Delta^{2k} \sigma_1 \sigma_2 \sigma_1\ $ for some $k \in \Z$
\item $\Delta^{2k }\sigma_1 \sigma_2 \sigma_1 \sigma_2\ $ for some $k \in \Z$.
\end{enumerate}

This form is unique up to cyclic permutation
of the word following $\Delta^{2k}$.  Since we are concerned with knots only and $\Delta^2$ is a pure braid, we need not consider cases $(2)$ and $(4)$.

In each of the remaining cases, the 3-braid is a product of $ \Delta^{2k}$ and a homogeneous braid. The closures of homogeneous braids
are fibered by \cite{stallings}. Thus in each case the closure of the 3-braid is a twisted fibered braid. 
Corollary \ref{cor:twist1} applies to give the desired conclusion. 
\end{proof}

\bibliographystyle{amsplain}
\bibliography{references}

\end{document}